\documentclass{article}

\usepackage{amsmath}
\usepackage{amsthm}
\usepackage{alltt}
\usepackage{hyperref}

\newtheorem{theorem}{Theorem}[section]

\numberwithin{equation}{section}

\newcommand{\res}[1]
{\textrm{Res}_{#1}}
\newcommand{\cycle}[2]
{\genfrac{[}{]}{0pt}{}{#1}{#2}}

\title{A Generalization of the Chu-Vandermonde Convolution and some Harmonic Number Identities}
\author{M.J. Kronenburg}
\date{}

\begin{document}

\maketitle

\begin{abstract}
A generalization of the Chu-Vandermonde convolution is presented and proved
with the integral representation method.
This identity can be transformed into another identity, which has as special cases
two known identities.
Another identity that is closely related to this identity is presented and proved.
Using the modified geometric series from another paper, some closely related
identities are listed.
Some corresponding harmonic number identities are derived,
which have as special cases some known harmonic number identities.
For one combinatorial sum a recursion formula is derived and used
to compute a few examples.

\end{abstract}

\noindent
\textbf{Keywords}: binomial coefficient, combinatorial identities,
harmonic number.\\
\textbf{MSC 2010}: 05A10, 05A19

\section{A Generalization of the Chu-Vandermonde\\ Convolution}

The following theorem is a generalization of the Chu-Vandermonde convolution.
\begin{theorem}
\begin{equation}\label{chugen}
 \sum_{k=0}^n \binom{a}{k}\binom{b}{n-k}\binom{k}{c}\binom{n-k}{d} 
  = \binom{a+b-c-d}{n-c-d}\binom{a}{c}\binom{b}{d}
\end{equation}
\end{theorem}
\begin{proof}
Applying the trinomial revision identity \cite{GKP94,K97,K15} to the summand twice,
the identity simplifies to:
\begin{equation}
 \sum_{k=0}^n \binom{a-c}{k-c}\binom{b-d}{n-d-k} = \binom{a+b-c-d}{n-c-d}
\end{equation}
Using the integral representation method for combinatorial sums \cite{CB84,E84},
this becomes:\\
\begin{equation}
\begin{split}
 & \sum_{k=0}^{\infty} \res{x}\frac{(1+x)^{a-c}}{x^{k-c+1}}\res{y}\frac{(1+y)^{b-d}}{y^{n-d-k+1}} \\
 ={} & \res{x}\res{y} \frac{(1+x)^{a-c}(1+y)^{b-d}}{x^{-c+1}y^{n-d+1}} \sum_{k=0}^{\infty}\left(\frac{y}{x}\right)^k \\
 ={} & \res{x}\res{y} \frac{(1+x)^{a-c}(1+y)^{b-d}x^{c-1}}{y^{n-d+1}(1-y/x)} \\
 ={} & \res{y}\res{x} \frac{(1+x)^{a-c}(1+y)^{b-d}x^c}{y^{n-d+1}(x-y)} \\
 ={} & \res{y} \frac{(1+y)^{a+b-c-d}}{y^{n-c-d+1}} \\
 ={} & \binom{a+b-c-d}{n-c-d}
\end{split}
\end{equation}
\end{proof}
For this identity, the lower limit of the summation may be replaced by $\max(n-b,c)$ and the upper limit
by $\min(a,n-d)$.
The special case $c=d=0$ reduces to the Chu-Vandermonde convolution:
\begin{equation}\label{chuold}
 \sum_{k=0}^n \binom{a}{k}\binom{b}{n-k} = \binom{a+b}{n}
\end{equation}
and the special case $d=0$ was already known in literature \cite{G72,Q16}.
The special case $a=b=n$ reduces to:
\begin{equation}\label{sqrsum1}
 \sum_{k=0}^n \binom{n}{k}^2 \binom{k}{c}\binom{n-k}{d} = \binom{2n-c-d}{n}\binom{n}{c}\binom{n}{d}
\end{equation}
The following identity replaces a binomial coefficient by its symmetry equivalent \cite{GKP94,K97,K15}:
\begin{equation}\label{binomsym}
 \binom{n}{k} = (-1)^k\binom{-n+k-1}{k}
\end{equation}
Replacing $a$ by $-p-1$ and $b$ by $-q-1$, and using (\ref{binomsym}),
the identity transforms to:
\begin{equation}\label{chugen2}
 \sum_{k=0}^n \binom{p+k}{p}\binom{q+n-k}{q}\binom{k}{c}\binom{n-k}{d} 
  = \binom{n+p+q+1}{n-c-d}\binom{p+c}{c}\binom{q+d}{d}
\end{equation}
The special case $c=d=0$ reduces to \cite{G56,G72,WGW13}:
\begin{equation}
 \sum_{k=0}^n \binom{p+k}{p}\binom{q+n-k}{q} = \binom{n+p+q+1}{n}
\end{equation}
and the special case $p=q=0$ reduces to \cite{G72,GKP94,K97,S90}:
\begin{equation}
 \sum_{k=0}^n \binom{k}{c}\binom{n-k}{d} = \binom{n+1}{c+d+1}
\end{equation}
Another theorem that is closely related to theorem 1.1 is the following.
\begin{theorem}
\begin{equation}\label{chu2gen}
 \sum_{k=0}^a \binom{a}{k}\binom{b}{m+k}\binom{k}{c}\binom{m+k}{d} 
  = \binom{a+b-c-d}{m+a-d}\binom{a}{c}\binom{b}{d}
\end{equation}
\end{theorem}
\begin{proof}
Applying the trinomial revision identity \cite{GKP94,K97,K15} to the summand twice,
the identity simplifies to:
\begin{equation}
 \sum_{k=0}^a \binom{a-c}{k-c}\binom{b-d}{k+m-d} = \binom{a+b-c-d}{m+a-d}
\end{equation}
Using the integral representation method for combinatorial sums \cite{CB84,E84},
this becomes:
\begin{equation}
\begin{split}
 & \sum_{k=0}^{\infty} \res{x}\frac{(1+x)^{a-c}}{x^{k-c+1}}\res{y}\frac{(1+y)^{b-d}}{y^{k+m-d+1}} \\
 ={} & \res{x}\res{y} \frac{(1+x)^{a-c}(1+y)^{b-d}}{x^{-c+1}y^{m-d+1}} \sum_{k=0}^{\infty}\left(\frac{1}{x y}\right)^k \\
 ={} & \res{x}\res{y} \frac{(1+x)^{a-c}(1+y)^{b-d}x^{c-1}}{y^{m-d+1}(1-1/(x y))} \\
 ={} & \res{y}\res{x} \frac{(1+x)^{a-c}(1+y)^{b-d}x^c}{y^{m-d+1}(x-1/y)} \\
 ={} & \res{y} \frac{(1+1/y)^{a-c}(1+y)^{b-d}}{y^{m+c-d+1}} \\
 ={} & \res{y} \frac{(1+y)^{a+b-c-d}}{y^{m+a-d+1}} \\
 ={} & \binom{a+b-c-d}{m+a-d}
\end{split}
\end{equation}
\end{proof}
For this identity, the lower limit of the summation may be replaced by\linebreak $\max(c,d-m)$
and the upper limit by $\min(a,b-m)$.
The special case $c=d=0$ reduces to \cite{G72,Q16}:
\begin{equation}
 \sum_{k=0}^a\binom{a}{k}\binom{b}{m+k} = \binom{a+b}{m+a}
\end{equation}
The special case $m=0$ and $a=b=n$ reduces to:
\begin{equation}\label{sqrsum2}
 \sum_{k=0}^n \binom{n}{k}^2 \binom{k}{c}\binom{k}{d} = \binom{2n-c-d}{n-d}\binom{n}{c}\binom{n}{d}
\end{equation}
The special cases $c=d=0$, $c=1$, $d=0$ and $c=d=1$ of this formula are well known \cite{G72}.\\
Using the residue formula (2.2) and the modified geometric series (3.13) and (3.14) in \cite{K19} 
and the absorption identity (3.3) in \cite{K15}:
\begin{equation}
 \sum_{k=0}^n \binom{a}{k}\binom{b}{n-k}\binom{k}{c}\binom{n-k}{d} k
  =  \binom{a+b-c-d}{n-c-d}\binom{a}{c}\binom{b}{d} \frac{n(a-c)+bc-ad}{a+b-c-d}
\end{equation}
\begin{equation}
\begin{split}
 & \sum_{k=0}^n \binom{a}{k}\binom{b}{n-k}\binom{k}{c}\binom{n-k}{d} k^2
   = \binom{a+b-c-d}{n-c-d}\binom{a}{c}\binom{b}{d} \\
 & \cdot \frac{n(a-c)[n(a-c-1)+2(bc-ad)+b+d]+(bc-ad)^2-bd(a-c)-ac(b-d)}{(a+b-c-d)(a+b-c-d-1)} \\
\end{split}
\end{equation}
\begin{equation}
\begin{split}
 f_1(n,a,b,c,d) & = (a+b-c-d-1)(a+b-c-d-2) \\
 & \quad\cdot [c^3(a+b-c-d)+(a-c)(3c^2+3c+1)(n-c-d)] \\
 & + (a-c)(a-c-1)(n-c-d)(n-c-d-1) \\
 & \quad\cdot [3(a+b-c-d-2)(c+1)+(a-c-2)(n-c-d-2)] \\
\end{split}
\end{equation}
\begin{equation}
\begin{split}
 \sum_{k=0}^n \binom{a}{k}\binom{b}{n-k} & \binom{k}{c}\binom{n-k}{d} k^3
  = \binom{a+b-c-d}{n-c-d}\binom{a}{c}\binom{b}{d} \\
 & \cdot \frac{f_1(n,a,b,c,d)}{(a+b-c-d)(a+b-c-d-1)(a+b-c-d-2)} \\
\end{split}
\end{equation}
\begin{equation}
 \sum_{k=0}^a \binom{a}{k}\binom{b}{m+k}\binom{k}{c}\binom{m+k}{d} k
  =  \binom{a+b-c-d}{m+a-d}\binom{a}{c}\binom{b}{d} \frac{m(c-a)+ab-cd}{a+b-c-d}
\end{equation}
\begin{equation}
\begin{split}
 & \sum_{k=0}^a \binom{a}{k}\binom{b}{m+k}\binom{k}{c}\binom{m+k}{d} k^2 
 = \binom{a+b-c-d}{m+a-d}\binom{a}{c}\binom{b}{d} \\
 & \cdot \frac{m(a-c)[m(a-c-1)+2(cd-ab)+b+d]+(cd-ab)^2-bd(a-c)-ac(b-d)}{(a+b-c-d)(a+b-c-d-1)} \\
\end{split}
\end{equation}
\begin{equation}
\begin{split}
 f_2(m,a,b,c,d) & = (a+b-c-d-1)(a+b-c-d-2) \\
 & \quad\cdot [c^3(a+b-c-d)+(a-c)(3c^2+3c+1)(b-m-c)] \\
 & + (a-c)(a-c-1)(b-m-c)(b-m-c-1) \\
 & \quad\cdot [3(a+b-c-d-2)(c+1)+(a-c-2)(b-m-c-2)] \\
\end{split}
\end{equation}
\begin{equation}
\begin{split}
 \sum_{k=0}^a \binom{a}{k}\binom{b}{m+k} & \binom{k}{c}\binom{m+k}{d} k^3
 = \binom{a+b-c-d}{m+a-d}\binom{a}{c}\binom{b}{d} \\
 & \cdot \frac{f_2(m,a,b,c,d)}{(a+b-c-d)(a+b-c-d-1)(a+b-c-d-2)} \\
\end{split}
\end{equation}
\begin{equation}
\begin{split}
 \sum_{k=0}^n \binom{p+k}{p}\binom{q+n-k}{q} & \binom{k}{c}\binom{n-k}{d} k 
 =  \binom{n+p+q+1}{n-c-d}\binom{p+c}{c}\binom{q+d}{d} \\
 & \cdot \frac{n(p+c+1)+(q+1)c-(p+1)d}{p+q+c+d+2} \\
\end{split}
\end{equation}
\begin{equation}
\begin{split}
 f_3(n,p,q,c,d) = & n(p+c+1)\{n(p+c+2)+2[(q+1)c-(p+1)d]+q-d+1\} \\
 & + [(q+1)c-(p+1)d]^2-(q+1)c(p+d+1)-(p+1)d(q+c+1) \\
\end{split}
\end{equation}
\begin{equation}
\begin{split}
 \sum_{k=0}^n \binom{p+k}{p}\binom{q+n-k}{q} & \binom{k}{c}\binom{n-k}{d} k^2 
  = \binom{n+p+q+1}{n-c-d}\binom{p+c}{c}\binom{q+d}{d} \\
 & \cdot \frac{f_3(n,p,q,c,d)}{(p+q+c+d+2)(p+q+c+d+3)}  \\
\end{split}
\end{equation}
\begin{equation}
\begin{split}
 f_4(n,p,q,c,d) & = (p+q+c+d+3)(p+q+c+d+4) \\
 & \quad\cdot [c^3(p+q+c+d+2)+(p+c+1)(3c^2+3c+1)(n-c-d)] \\
 & + (p+c+1)(p+c+2)(n-c-d)(n-c-d-1) \\
 & \quad\cdot [3(p+q+c+d+4)(c+1)+(p+c+3)(n-c-d-2)] \\
\end{split}
\end{equation}
\begin{equation}
\begin{split}
 \sum_{k=0}^n \binom{p+k}{p} & \binom{q+n-k}{q}\binom{k}{c}\binom{n-k}{d} k^3 
  = \binom{n+p+q+1}{n-c-d}\binom{p+c}{c}\binom{q+d}{d} \\
 & \cdot \frac{f_4(n,p,q,c,d)}{(p+q+c+d+2)(p+q+c+d+3)(p+q+c+d+4)}  \\
\end{split}
\end{equation}

\section{Harmonic Number Identities}

The definition of the generalized harmonic numbers with
nonnegative integer $n$, complex order $m$ and complex offset $c$, is \cite{K11,L07}:
\begin{equation}\label{genharmdef}
 H_{c,n}^{(m)} = \sum_{k=1}^{n}\frac{1}{(c+k)^m}
\end{equation}
from which follows that $H_{c,0}^{(m)}=0$, and for notation $H_n^{(m)}=H_{0,n}^{(m)}$.
From this definition follows for nonnegative integer $c$:
\begin{equation}\label{genharmidef}
 H_{c,n}^{(m)} = H_{c+n}^{(m)} - H_c^{(m)}
\end{equation}
The classical harmonic numbers are:
\begin{equation}
 H_n = H_{0,n}^{(1)}
\end{equation}
Using $d/dx\Gamma(x)=\Gamma(x)\psi(x)$ where $\psi(x)$ is the digamma function,
and using\linebreak $\psi(x+n+1)-\psi(x+1)=H^{(1)}_{x,n}$ \cite{AAR99},
these harmonic numbers are linked to binomial coefficients:
\begin{equation}\label{binomdif}
 \frac{d}{dx}\binom{x+y}{n} = \binom{x+y}{n}H_{x+y-n,n}^{(1)} = \binom{x+y}{n}(H_{x+y}-H_{x+y-n})
\end{equation}
\begin{equation}\label{binomdif2}
 \frac{d}{dx}\binom{n}{x+y} = \binom{n}{x+y}H_{x+y,n-2(x+y)}^{(1)} = \binom{n}{x+y}(H_{n-(x+y)}-H_{x+y})
\end{equation}
For the generalized harmonic numbers (\ref{genharmdef}):
\begin{equation}\label{harmdif}
 \frac{d}{dx}H_{x+y,n}^{(m)} = -m H_{x+y,n}^{(m+1)}
\end{equation}
When differentiating finite summation terms, care must be taken that the differentiated symbol
is not present in the summation limits.
Because the argument of a classical harmonic number cannot be negative,
these harmonic numbers impose constraints on the parameters.
When there are additional constraints on the parameters they are mentioned.\\
Differentiating (\ref{chugen}) to $a$, the following identity for $a\geq n$ results:
\begin{equation}
\begin{split}
 & \sum_{k=0}^n \binom{a}{k}\binom{b}{n-k}\binom{k}{c}\binom{n-k}{d}H_{a-k} \\
={} & \binom{a+b-c-d}{n-c-d}\binom{a}{c}\binom{b}{d}
  (H_{a+b-n}-H_{a+b-c-d}+H_{a-c})
\end{split}
\end{equation}
Differentiating to $b$, the following identity for $b\geq n$ results:
\begin{equation}
\begin{split}
 & \sum_{k=0}^n \binom{a}{k}\binom{b}{n-k}\binom{k}{c}\binom{n-k}{d}H_{b-n+k} \\
={} & \binom{a+b-c-d}{n-c-d}\binom{a}{c}\binom{b}{d}
  (H_{a+b-n}-H_{a+b-c-d}+H_{b-d})
\end{split}
\end{equation}
Replacing $k$ by $n-k$ and interchanging $a$ with $b$ and $c$ with $d$,
these two identities are equivalent.
Differentiating to $a$ and $b$, the following identity for $a\geq n$ and $b\geq n$ results:
\begin{equation}
\begin{split}
 & \sum_{k=0}^n \binom{a}{k}\binom{b}{n-k}\binom{k}{c}\binom{n-k}{d}H_{a-k}H_{b-n+k} \\
 ={} & \binom{a+b-c-d}{n-c-d}\binom{a}{c}\binom{b}{d} [H^{(2)}_{a+b-n}-H^{(2)}_{a+b-c-d} \\
 & + (H_{a+b-n}-H_{a+b-c-d}+H_{a-c})(H_{a+b-n}-H_{a+b-c-d}+H_{b-d})]
\end{split}
\end{equation}
The special case $a=b=n$ and $c=d=0$ reduces to the known identities \cite{CC09,CD05,G72,L07,WGW13}:
\begin{equation}
 \sum_{k=0}^n \binom{n}{k}^2 H_k = \binom{2n}{n}(2H_n-H_{2n})
\end{equation}
\begin{equation}
 \sum_{k=0}^n \binom{n}{k}^2 H_kH_{n-k} = \binom{2n}{n}[H^{(2)}_n-H^{(2)}_{2n}+(2H_n-H_{2n})^2]
\end{equation}
Differentiating (\ref{chugen}) to $c$, the following identity for $b\leq n$ results:
\begin{equation}
  \sum_{k=0}^n \binom{a}{k}\binom{b}{n-k}\binom{n-k}{d}H_k
= \binom{a+b-d}{n-d}\binom{b}{d} (H_{n-d}-H_{a+b-d}+H_a)
\end{equation}
Differentiating to $d$, the following identity for $a\leq n$ results:
\begin{equation}
 \sum_{k=0}^n \binom{a}{k}\binom{b}{n-k}\binom{k}{c}H_{n-k}
 = \binom{a+b-c}{n-c}\binom{a}{c} (H_{n-c}-H_{a+b-c}+H_b)
\end{equation}
Replacing $k$ by $n-k$ and interchanging $a$ with $b$ and $c$ with $d$,
these two identities are equivalent.
Differentiating to $c$ and $d$, the following identity for $a\leq n$ and $b\leq n$ results:
\begin{equation}
\begin{split}
 & \sum_{k=0}^n \binom{a}{k}\binom{b}{n-k}H_{k}H_{n-k} \\
 ={} & \binom{a+b}{n} [H^{(2)}_n-H^{(2)}_{a+b}
  + (H_n-H_{a+b}+H_a)(H_n-H_{a+b}+H_b)]
\end{split}
\end{equation}
In (\ref{chugen}) replacing $k$ by $c+k$ and $n$ by $n+c$, and
differentiating to $c$, the following identity for $a-c\geq n$ results:
\begin{equation}
\begin{split}
  & \sum_{k=0}^n \binom{a}{c+k}\binom{b}{n-k}\binom{c+k}{c}\binom{n-k}{d}H_{a-c-k} \\
={} & \binom{a+b-c-d}{n-d}\binom{a}{c}\binom{b}{d} (H_{a+b-n-c}-H_{a+b-c-d}+H_{a-c})
\end{split}
\end{equation}
Replacing $n$ with $n+d$ and differentiating to $d$, the following identity for $b-d\geq n$ results:
\begin{equation}
\begin{split}
 & \sum_{k=0}^n \binom{a}{k}\binom{b}{n+d-k}\binom{k}{c}\binom{n+d-k}{d}H_{b-d-n+k} \\
 ={} & \binom{a+b-c-d}{n-c}\binom{a}{c}\binom{b}{d} (H_{a+b-n-d}-H_{a+b-c-d}+H_{b-d})
\end{split}
\end{equation}
Replacing $k$ by $n-k$ and interchanging $a$ with $b$ and $c$ with $d$,
these two identities are equivalent.
Differentiating to $c$ and $d$, the following identity for $a-c\geq n$ and $b-d\geq n$ results:
\begin{equation}
\begin{split}
 & \sum_{k=0}^n \binom{a}{c+k}\binom{b}{n+d-k}\binom{c+k}{c}\binom{n+d-k}{d}H_{a-c-k}H_{b-d-n+k} \\
 ={} & \binom{a+b-c-d}{n}\binom{a}{c}\binom{b}{d} [H^{(2)}_{a+b-n-c-d}-H^{(2)}_{a+b-c-d} \\
 & + (H_{a+b-n-c-d}-H_{a+b-c-d}+H_{a-c})(H_{a+b-n-c-d}-H_{a+b-c-d}+H_{b-d})]
\end{split}
\end{equation}
Differentiating (\ref{chugen2}) to $p$, the following identity results:
\begin{equation}
\begin{split}
 & \sum_{k=0}^n \binom{p+k}{p}\binom{q+n-k}{q}\binom{k}{c}\binom{n-k}{d}H_{p+k} \\
={} & \binom{n+p+q+1}{n-c-d}\binom{p+c}{c}\binom{q+d}{d}
  (H_{n+p+q+1}-H_{p+q+c+d+1}+H_{p+c})
\end{split}
\end{equation}
The special case $p=q=0$ is found in \cite{S90},
and the special case $c=d=0$ is found in \cite{WGW13}.
Differentiating to $q$, the following identity results:
\begin{equation}
\begin{split}
 & \sum_{k=0}^n \binom{p+k}{p}\binom{q+n-k}{q}\binom{k}{c}\binom{n-k}{d}H_{q+n-k} \\
={} & \binom{n+p+q+1}{n-c-d}\binom{p+c}{c}\binom{q+d}{d}
  (H_{n+p+q+1}-H_{p+q+c+d+1}+H_{q+d})
\end{split}
\end{equation}
Replacing $k$ by $n-k$ and interchanging $p$ with $q$ and $c$ with $d$,
these two identities are equivalent.
Differentiating to $p$ and $q$ the following identity results:
\begin{equation}
\begin{split}
 & \sum_{k=0}^n \binom{p+k}{p}\binom{q+n-k}{q}\binom{k}{c}\binom{n-k}{d}H_{p+k}H_{q+n-k} \\
={} & \binom{n+p+q+1}{n-c-d}\binom{p+c}{c}\binom{q+d}{d}
   [H_{p+q+c+d+1}^{(2)}-H_{n+p+q+1}^{(2)} \\
  & +(H_{n+p+q+1}-H_{p+q+c+d+1}+H_{p+c})(H_{n+p+q+1}-H_{p+q+c+d+1}+H_{q+d})]
\end{split}
\end{equation}
The special case $p=q=0$ is found in \cite{S90},
and the special case $c=d=0$ is found in \cite{WGW13}.
In (\ref{chugen2}) replacing $k$ by $c+k$ and $n$ by $n+c$, and
differentiating to $c$, the following identity results:
\begin{equation}
\begin{split}
 & \sum_{k=0}^n \binom{p+c+k}{p}\binom{q+n-k}{q}\binom{c+k}{c}\binom{n-k}{d}H_{p+c+k} \\
={} & \binom{n+p+q+c+1}{n-d}\binom{p+c}{c}\binom{q+d}{d}
  (H_{n+p+q+c+1}-H_{p+q+c+d+1}+H_{p+c})
\end{split}
\end{equation}
Replacing $n$ with $n+d$ and differentiating to $d$, the following identity results:
\begin{equation}
\begin{split}
 & \sum_{k=0}^n \binom{p+k}{p}\binom{q+n+d-k}{q}\binom{k}{c}\binom{n+d-k}{d}H_{q+n+d-k} \\
={} & \binom{n+p+q+d+1}{n-c}\binom{p+c}{c}\binom{q+d}{d}
  (H_{n+p+q+d+1}-H_{p+q+c+d+1}+H_{q+d})
\end{split}
\end{equation}
Replacing $k$ by $n-k$ and interchanging $p$ with $q$ and $c$ with $d$,
these two identities are equivalent.
Differentiating to $c$ and $d$ the following identity results:
\begin{equation}
\begin{split}
 & \sum_{k=0}^n \binom{p+c+k}{p}\binom{q+n+d-k}{q}\binom{c+k}{c}\binom{n+d-k}{d}H_{p+c+k}H_{q+n+d-k} \\
& ={} \binom{n+p+q+c+d+1}{n}\binom{p+c}{c}\binom{q+d}{d}
   [H_{p+q+c+d+1}^{(2)}-H_{n+p+q+c+d+1}^{(2)} \\
  & +(H_{n+p+q+c+d+1}-H_{p+q+c+d+1}+H_{p+c})(H_{n+p+q+c+d+1}-H_{p+q+c+d+1}+H_{q+d})]
\end{split}
\end{equation}
Differentiating (\ref{chu2gen}) to $a$, the following identity for $a\geq b-m$ results:
\begin{equation}
\begin{split}
 & \sum_{k=0}^{b-m} \binom{a}{k}\binom{b}{m+k}\binom{k}{c}\binom{m+k}{d}H_{a-k} \\
={} & \binom{a+b-c-d}{m+a-d}\binom{a}{c}\binom{b}{d}(H_{m+a-d}-H_{a+b-c-d}+H_{a-c})
\end{split}
\end{equation}
Differentiating to $b$, the following identity for $b\geq a+m$ results:
\begin{equation}
\begin{split}
 & \sum_{k=0}^a \binom{a}{k}\binom{b}{m+k}\binom{k}{c}\binom{m+k}{d}H_{b-m-k} \\
={} & \binom{a+b-c-d}{m+a-d}\binom{a}{c}\binom{b}{d}(H_{b-m-c}-H_{a+b-c-d}+H_{b-d})
\end{split}
\end{equation}
Differentiating to $c$, the following identity for $d\geq m$ results:
\begin{equation}
\begin{split}
 & \sum_{k=0}^a \binom{a}{k}\binom{b}{m+k}\binom{m+k}{d}H_k \\
={} & \binom{a+b-d}{m+a-d}\binom{b}{d}(H_a-H_{a+b-d}+H_{b-m})
\end{split}
\end{equation}
Differentiating to $d$, the following identity with $m=d$ results:
\begin{equation}
\begin{split}
 & \sum_{k=0}^a \binom{a}{k}\binom{b}{d+k}\binom{k}{c}\binom{d+k}{d}H_k \\
={} & \binom{a+b-c-d}{a}\binom{a}{c}\binom{b}{d}(H_a-H_{a+b-c-d}+H_{b-d})
\end{split}
\end{equation}
Differentiating to $a$ and $c$, the following identity for $a\geq b-m$ and $d\geq m$ results:
\begin{equation}
\begin{split}
 & \sum_{k=0}^{b-m} \binom{a}{k}\binom{b}{m+k}\binom{m+k}{d}H_k H_{a-k} \\
={} & \binom{a+b-d}{m+a-d}\binom{b}{d}[H_a^{(2)}-H_{a+b-d}^{(2)} \\
 & + (H_a-H_{a+b-d}+H_{m+a-d})(H_a-H_{a+b-d}+H_{b-m})]
\end{split}
\end{equation}
Differentiating to $a$ and $d$, the following identity for $a\geq b-d$ and $m=d$ results:
\begin{equation}
\begin{split}
 & \sum_{k=0}^{b-d} \binom{a}{k}\binom{b}{d+k}\binom{k}{c}\binom{d+k}{d}H_k H_{a-k} \\
={} & \binom{a+b-c-d}{a}\binom{a}{c}\binom{b}{d}[H_a^{(2)}-H_{a+b-c-d}^{(2)} \\
 & + (H_a-H_{a+b-c-d}+H_{a-c})(H_a-H_{a+b-c-d}+H_{b-d})]
\end{split}
\end{equation}
Differentiating to $b$ and $c$, the following identity for $b\geq a+m$ and $d\geq m$ results:
\begin{equation}
\begin{split}
 & \sum_{k=0}^a \binom{a}{k}\binom{b}{m+k}\binom{m+k}{d}H_k H_{b-m-k} \\
={} & \binom{a+b-d}{m+a-d}\binom{b}{d}[H_{b-m}^{(2)}-H_{a+b-d}^{(2)} \\
 & + (H_{b-m}-H_{a+b-d}+H_{b-d})(H_{b-m}-H_{a+b-d}+H_a)]
\end{split}
\end{equation}
Differentiating to $b$ and $d$, the following identity for $b\geq a+d$ and $m=d$ results:
\begin{equation}
\begin{split}
 & \sum_{k=0}^a \binom{a}{k}\binom{b}{d+k}\binom{k}{c}\binom{d+k}{d}H_k H_{b-d-k} \\
={} & \binom{a+b-c-d}{a}\binom{a}{c}\binom{b}{d}[H_{b-d}^{(2)}-H_{a+b-c-d}^{(2)} \\
 & + (H_{b-d}-H_{a+b-c-d}+H_{b-d-c})(H_{b-d}-H_{a+b-c-d}+H_a)]
\end{split}
\end{equation}

\section{A Recursion Formula for a Combinatorial Sum}

Using a recursion formula, a rational function $P_m(n)$ is found such that:
\begin{equation}
 \sum_{k=0}^n \binom{n}{k}^2 k^m = \binom{2n}{n}P_m(n)
\end{equation}
From (\ref{sqrsum1}) or (\ref{sqrsum2}) with $d=0$ we have:
\begin{equation}
 \sum_{k=0}^n \binom{n}{k}^2 \binom{k}{m} = \binom{2n-m}{n}\binom{n}{m}
\end{equation}
This formula is rewritten as:
\begin{equation}
 \sum_{k=0}^n \binom{n}{k}^2 \prod_{j=0}^{m-1}(k-j) = \binom{2n}{n}\prod_{j=0}^{m-1}\frac{(n-j)^2}{2n-j}
\end{equation}
Now the following is used:
\begin{equation}
 \prod_{j=0}^{m-1}(k-j) = \sum_{j=0}^{m}(-1)^{m-j}\cycle{m}{j}k^j
\end{equation}
where $\cycle{a}{b}$ is the Stirling number of the first kind \cite{GKP94}.
Then it is clear that $P_m(n)$ has the following recursion formula:
\begin{equation}
 P_m(n) = \prod_{k=0}^{m-1}\frac{(n-k)^2}{2n-k} - \sum_{k=0}^{m-1}(-1)^{m-k}\cycle{m}{k}P_k(n)
\end{equation}
The following are a few examples:
\begin{equation}
 \sum_{k=0}^n \binom{n}{k}^2 = \binom{2n}{n}
\end{equation}
\begin{equation}
 \sum_{k=0}^n \binom{n}{k}^2 k = \binom{2n}{n}\frac{n}{2}
\end{equation}
\begin{equation}
 \sum_{k=0}^n \binom{n}{k}^2 k^2 = \binom{2n}{n} \frac{n^3}{2(2n-1)}
\end{equation}
\begin{equation}
 \sum_{k=0}^n \binom{n}{k}^2 k^3 = \binom{2n}{n} \frac{n^3(n+1)}{4(2n-1)}
\end{equation}
\begin{equation}
 \sum_{k=0}^n \binom{n}{k}^2 k^4 = \binom{2n}{n} \frac{n^3(n^3+n^2-3n-1)}{4(2n-1)(2n-3)}
\end{equation}
\begin{equation}
 \sum_{k=0}^n \binom{n}{k}^2 k^5 = \binom{2n}{n} \frac{n^4(n+1)(n^2+2n-5)}{8(2n-1)(2n-3)}
\end{equation}
\begin{equation}
 \sum_{k=0}^n \binom{n}{k}^2 k^6 = \binom{2n}{n} \frac{n^3(n^6+3n^5-13n^4-15n^3+30n^2+8n-2)}{8(2n-1)(2n-3)(2n-5)}
\end{equation}
\begin{equation}
 \sum_{k=0}^n \binom{n}{k}^2 k^7 = \binom{2n}{n} \frac{n^4(n+1)(n^5+5n^4-15n^3-35n^2+70n-14)}{16(2n-1)(2n-3)(2n-5)}
\end{equation}

The Mathematica$^{\textregistered}$ \cite{W03} program used to compute the expressions
is given below:
\begin{alltt}
P[0]=1;
P[m_]:=P[m]=Factor[Simplify[Product[(n-k)^2/(2n-k),\{k,0,m-1\}]
 -Sum[StirlingS1[m,k]P[k],\{k,0,m-1\}]]]
\end{alltt}

\pdfbookmark[0]{References}{}

\end{document}